\documentclass{birkau}

\usepackage{amsmath,amssymb,url}
\usepackage{amssymb}
\usepackage{graphicx}
\usepackage[usenames]{color}
\usepackage [all]{xy}
\usepackage{mathrsfs}
\usepackage{mathtools}
\usepackage{forest}
\usepackage{stmaryrd}
\usepackage{lmodern}

\numberwithin{equation}{section}

\theoremstyle{plain}
\newtheorem{theorem}{Theorem}[section]
\newtheorem{lemma}[theorem]{Lemma}
\newtheorem{proposition}[theorem]{Proposition}
\newtheorem{corollary}[theorem]{Corollary}

\theoremstyle{definition}
\newtheorem{definition}[theorem]{Definition}

\newtheorem{example}[theorem]{Example}
\newtheorem*{note}{Note} 

\DeclareMathOperator{\im}{im}

\DeclareMathOperator{\length}{length}
\DeclareMathOperator{\lcm}{lcm}
\DeclareMathOperator{\Lab}{Lab}
\DeclareMathOperator{\rk}{rk}
\DeclareMathOperator{\width}{width}

\newcommand{\com}{\mathrm{\complement}}
\newcommand{\Sub}{\mathbf{Sub}}
\newcommand{\Spec}{\mathbf{Spec}_+}

\newcommand{\M}{\mathbb{M}}
\newcommand{\N}{\mathscr{N}}
\newcommand{\Prime}{\mathscr{P}}
\newcommand{\G}{\mathcal{G}}

\newcommand{\p}{ \mathfrak{p}}
\newcommand{\q}{ \mathfrak{q}}
\newcommand{\rp}{ \mathfrak{r}}
\newcommand{\s}{ \mathfrak{s}}
\newcommand{\tc}{ \mathfrak{t}}

\newcommand{\primerexemple}{{\fontsize{2.5}{4}\selectfont
\begin{forest}
 for tree={circle,draw,l sep=1pt, s sep=10pt}
[ [ [][]  ][ [][ [][] ] ]]
\end{forest}
}}
\newcommand{\un}{{\fontsize{2.5}{4}\selectfont
\begin{forest}
 for tree={circle,draw,l sep=1pt, s sep=10pt}
[ ]
\end{forest}
}}
\newcommand{\dos}{{\fontsize{2.5}{4}\selectfont
\begin{forest}
 for tree={circle,draw,l sep=1pt, s sep=10pt}
[ [][] ]
\end{forest}
}}
\newcommand{\trespositiu}{{\fontsize{2.5}{4}\selectfont
\begin{forest}
 for tree={circle,draw,l sep=1pt, s sep=10pt}
[ [][  [] [ ] ] ]
\end{forest}
}}
\newcommand{\tresnegatiu}{{\fontsize{2.5}{4}\selectfont
\begin{forest}
 for tree={circle,draw,l sep=1pt, s sep=10pt}
[ [ [] [] ] [] ]
\end{forest}
}}
\newcommand{\quatrepositiu}{{\fontsize{2.5}{4}\selectfont
\begin{forest}
 for tree={circle,draw,l sep=1pt, s sep=10pt}
[ [][  [] [ [] [] ] ] ]
\end{forest}
}}
\newcommand{\quatrenegatiu}{{\fontsize{2.5}{4}\selectfont
\begin{forest}
 for tree={circle,draw,l sep=1pt, s sep=10pt}
[ [ [ [] [] ] [] ] [] ]
\end{forest}
}}
\newcommand{\producte}{{\fontsize{2.5}{4}\selectfont
\begin{forest}
 for tree={circle,draw,l sep=1pt, s sep=10pt}
[ [ [] []]   [  [[][]] [[][]] ]  ]
\end{forest}
}}
\newcommand{\dosdos}{{\fontsize{2.5}{4}\selectfont
\begin{forest}
 for tree={circle,draw,l sep=1pt, s sep=10pt}
[ [[][]] [[][]]    ]
\end{forest}
}}
\newcommand{\dosdosdos}{{\fontsize{2.5}{4}\selectfont
\begin{forest}
 for tree={circle,draw,l sep=1pt, s sep=10pt}
[ [[[][]]       [[][]]] [[[][]]     [[][]]]    ]
\end{forest}
}}

\begin{document}


\title[Arithmetic and $k$-maximality of the cyclic free magma]{Arithmetic and $k$-maximality \\ of the cyclic free magma}

\author[C. Card\'o]{Carles Card\'o}
\address{Departament de Ci\`encies de la Computaci\'o, \\Campus Nord, Edifici Omega, Jordi Girona Salgado 1-3. \\Universitat Polit\`ecnica de Catalunya \\ 08034, Barcelona, \\Catalonia (Spain)}
\urladdr{http://www.cs.upc.edu}
\email{cardocarles@gmail.com}

\thanks{This research was supported the recognition 2017SGR-856 (MACDA) from AGAUR (Generalitat de Catalunya).}




\begin{abstract}
We survey free magmas and we explore the structure of their submagmas. 
By equipping the cyclic free magma with a second distributive operation we obtain a ringoid-like structure with some primitive arithmetical properties. 
A submagma is $k$-maximal when there are only $k-1$ submagmas between it and the free magma itself. 
These two tools, arithmetic and maximality, allow us to study the lattice of the submagmas of a free magma.
\end{abstract}

\maketitle


\section {Introduction}

Free Magmas are objects of interest for computer science in so far as their elements are identified with binary trees. See, for example, \cite{knuth1968TheArt, Sedgewick1995Analysis} for a general perspective, or \cite{flajolet1979TheNumber, Amerlynck1998Iterees} for an instance of specific application.  
 Certain subsets of a free magma, \emph{Hall sets} and \emph{Lazard sets}, are used to define a basis of a free Lie algebra, see \cite{Reutenauer1993BookFreeLie}. 
But the main theoretical feature of free magmas relies on the fact that any structure with only one operation can be obtained from a free magma via quotient, \cite{gratzer2008universal, sankappanavar1981course}. 
In spite of computational applications or the utility in Lie algebras, the interest of the free magmas seems to be relegated to that fundamental fact, and algebra textbooks usually introduce them only as a preamble for subsequent more standard structures such as semigroups, monoids or groups. 

Resuming the spirit of Blondel, \cite{blondel1995UneFamille,blondel1994Properties}, the cyclic free magma, 
which we will notate here as $\M$, should have the same status as $\mathbb{N}$, the cyclic free monoid, or as $\mathbb{Z}$, the cyclic free group. In fact Blondel propose to call the elements of the cyclic free magma \emph{structured numbers} (denomination which will not be used here). Unfortunately, it seems that there are no publications equating $\M$ to $\mathbb{N}$ or $\mathbb{Z}$ regarding the structure of its submagmas. The focus of the cited articles is put on the definable operations over trees and the treatment is rather computational than algebraic. 

The lattice of the subgroups of the cyclic free group is simple and beautiful: every subgroup is of the form $a\mathbb{Z}$, with the operations $a\mathbb{Z} + b\mathbb{Z}=\gcd(a,b)\mathbb{Z}$ and $a\mathbb{Z} \cap b\mathbb{Z}=\lcm(a,b)\mathbb{Z}$. 
One could think that since the operation in $\M$ is so simple (in terms of trees, the result of merging two trees by a new root), the lattice of its submagmas, $\Sub(\M)$, should not be too complex, and hence, not too interesting. 
This is only partially true. It is easy to prove that a submagma of a free magma is always free, and that two free magmas are isomorphic iff their generator sets have the same size. However, the set of submagmas as a whole is not so trivial as in the case of integers. By way of example, $\Sub(\M)$ is not countable, not distributive, nor modular.  

In order to study $\Sub(\M)$ we need to review some basic facts which make the first sections deal with elementary algebraic notions. Thus, this part of the article can be considered a brief survey on free magmas with an arithmetical perspective and some additional new results. 
More specifically, first we equip $\M$ with a product operation which makes $\M$ resemble a ring with some primitive arithmetic properties and which will allow us to manipulate submagmas. Such operation has been already considered earlier, and actually \cite
{blondel1995Structured, blondel1995Operations, blondel1995UneFamille,blondel1994Properties, Duchon1979Right} define operations of upper degree over binary trees, such as exponentiation and further, which we will not deal with here.  
Secondly, we study the $k$-maximal submagmas of $\M$. We say that a submagma $M \subseteq \M$ is $k$-maximal when there are only $k-1$ submagmas between $M$ and $\M$. The knowledge of the arithmetic and the $k$-maximality offers us a first picture of $\M$ similar to that of $\mathbb{Z}$. 
  
The structure of the paper is as follows. Section~\ref{FreeMagmas} establishes notation and some elementary well-known facts on free magmas. The first contribution of this article is a characterization of a free magma as a graded and unidecomposable magma. 
Section~\ref{TheArithmetic} introduces arithmetical properties, and we show some results concerning the number of prime elements of a given length. Section~\ref{SubmagmasAndGenerators} examines some submagmas and their generating sets. We show that $\Sub(\M)$ is not countable. Section~\ref{IdealsAndPrincipalIdeals} studies the multiplicative ideals. We compare principal ideals of $\M$ to those of the  integers. We show that $\Sub(\M)$ is not modular. The main contribution is given in Section~\ref{MaximalAndAdditivePrime} and Section~\ref{AllTheFinitePrimeSets} which introduces maximality and the additive prime sets. We exhibit a method to calculate all of them. This permits the identification of the first positions of $\Sub(\M)$.

\section {Free magmas} \label{FreeMagmas}

We notate $\mathbb{N}$ the set of strictly positive integers. $2^X$ is the power set of the set $X$. Here we prefer the term \emph{magma} according to Serre and Bourbaki \cite[p.~1]{Bourbaki1989Algebra}, over the alternative \emph{groupoid} \cite[p.~88]{Rosenfeld1968Algebraic} which sometimes refers to a generalization of group notion in a categorial sense \cite{Brown1987From}. 
\begin{definition} \label{DefMagma} A \emph{magma} is a set $M$ with an operation (or \emph{law of composition}) $+:M\times M \longrightarrow M$. 
A \emph{gradation} of a magma $M$ is a mapping $\ell: M \longrightarrow \mathbb{N}$ such that $\ell(x+y)=\ell(x)+\ell(y)$. A magma is \emph{graded} iff it can be equipped with a gradation. 

A \emph{submagma} of a magma $M$ is a subset which is stable for the operation. We notate $\Sub(M)$ the set of submagmas of $M$. For convenience we will assume $\emptyset \in \Sub(M)$. 
\end{definition} 

Given a subset $X\subseteq M$ of the magma $M$, its \emph{generated magma} $\langle X \rangle$ is the least submagma in $M$ containing $X$.  
 We say that a set $X\subseteq M$ is a \emph{minimal generating set} iff for each subset $Y\subsetneq X$ we have $\langle Y \rangle \subsetneq \langle X\rangle$. When a magma $M$ is given by the context, we notate $A^\com=M\setminus A$. 
Notions and notation on homomorphisms, direct products and other algebraic concepts are defined as it is usual. 
\begin{definition} \label{FreeMagma} Given a non-empty set $X$ consider the sequence of sets:
\begin{align*}
X_1&=X,\\
X_{n}&=\bigcup_{k=1}^{n-1} X_k \times X_{n-k}, \mbox{ for } n>1.
\end{align*}
We define the \emph{free magma on} $X$, $(\M_X;+)$, as $\M_X=\bigcup_{n=1}^\infty X_k$ and $x+y=(x,y)$ for each $x,y\in \M_X$. In general, we say that a magma $M$ is \emph{free} iff there is some set $X$ such that $M\cong \M_X$. 
\end{definition}

We have adopted the definition of Bourbaki \cite[p.~81]{Bourbaki1989Algebra}. 
Of course this definition is equivalent to define freeness by satisfying the \emph{universal mapping property}, see \cite[p.~71]{sankappanavar1981course}. 
$X$ is the unique minimal generator set of the free magma $\M_X$. Given a magma $M$, any homomorphism of magmas $f:\M_X \longrightarrow M$ stays defined by the image of the generators.  
If $\langle X\rangle =M$, we have that the mapping $f:\M_X\longrightarrow M$ defined over the generators as $f(x)=x$ is an epimorphism. By the isomorphism theorem, $\M_X/\ker f \cong M$. Thus, any magma can be obtained from a free magma via quotient. 
All above pertains to the well-known folklore of algebra, see for example \cite{sankappanavar1981course, gratzer2008universal}.

Since the sets $X_i$ are disjoint, the magma $\M_X$ can be naturally graded with the mapping $\ell_X: \M_X \longrightarrow \mathbb{N}$ defined as $\ell_X(x)=n$ iff $x \in X_n$. Notice that $\ell_X(x)=1$ for all $x\in X$.
We have that $\M_X\cong \M_Y$ iff $X$ and $Y$ are bijectable, which can be proved straightforward from the definition using the gradation $\ell_X$.  
We will notate $\M=\M_{\{1\}}$ the cyclic free magma. Usually free magmas on $X$ are notated $M(X)$, however, since we are going to work a lot with the cyclic free magma, we prefer the shortest notation above. 
We call \emph{length} the gradation of $\M$ and we will notate it as $\ell=\ell_{\{1\}}$. 
There is a natural partial order in $\M$. Given the relation:
\begin{align*}z \rightarrow z' \iff z=x+y\,\, \mbox{ and }\,\, z'=x \mbox{ or } z'=y, \end{align*}
we define $\geq$ as the transitive and reflexive closure of $\rightarrow$. This satisfies that for all $x,y \in \M$: 
\begin{align*}
1\leq x;\,\,\,\,\,\,\,\,\, x, y \leq x+y;\,\,\,\,\,\,\,\,\, x\leq y \implies \ell(x)\leq \ell(y). 
\end{align*}
Elements in $\M$ are of the form: $1, 1+1, 1+(1+1), (1+1)+(1+1), \ldots$.
We fix the following abbreviations:
\begin{align*}2=1+1, \,\,\,\,\,\, 2^2=2+2, \,\,\,\,\,\, 2^3=2^2+2^2, \,\,\,\,\,\, 2^4=2^3+2^3,\ldots \end{align*}
with the convention that $2^0=1$ and $2^1=2$ (see later the Section~\ref{TheProduct} to understand this notation). 
For each $n> 1$ we define recursively:
\begin{align*}1_-=1, \,\,\,\, (n+1)_-=n_-+1, \,\,\,\,\mbox{and}\,\,\,\, 1_+=1, \,\,\,\, (n+1)_+=1+n_+.\end{align*}
Notice that $2_-=2=2_+$. 
Elements of $\M$ are identified with binary trees. Thus:
\begin{align*}1=\un \, , \,\,\,\,\,\,\,\,\, 2= \dos  \, , \,\,\,\,\,\,\,\,\, 3_-=\tresnegatiu \, , \,\,\,\,\,\,\,\,\, 3_+=\trespositiu \, , \end{align*}
\begin{align*} 4_-=\quatrenegatiu, \,\,\,\,\,\,\,\,\,\,\,\,\,\,\, 4_+=\quatrepositiu \, , \, \ldots \end{align*}
The sum $+$ is the union of the trees by a new root:
\begin{align*}2+3_+=(1+1)+(1+(1+1))=\Big( \dos \Big) + \Big( \, \trespositiu \Big) = \primerexemple \, .\end{align*}
Thus, powers of 2 coincide with full binary trees:
\begin{align*}2=\dos \, , \,\,\,\,\, 2^2=\dosdos \, , \,\,\,\,\, 2^3=\dosdosdos \, , \, \ldots\end{align*}
Given a set of $X\subseteq \M$, we notate $ (X)_n =\{ x\in X \mid \ell(x)=n\}$. For example:
\begin{align*}
(\M)_1&=\{1\},\\
(\M)_2&=\{2\},\\
(\M)_3&=\{3_-,\, 3_+\},\\
(\M)_4&=\{4_-,\, 1{+}3_-,\, 2^2,\, 3_+{+}1,\, 4_+\},\\
(\M)_5&=\{5_-,\, 1{+}4_-,\, 1{+}(1{+}3_-),\, (1{+}3_-){+}1,\, 2{+}3_-, \, 3_-{+}2,\, 2^2{+}1,\\
&\,\,\,\,\,\,\,\,\,\,\,\, 1{+}2^2, \, 2{+}3_+,\, 3_+{+}2,\, 1{+}(3_+{+}1),\,  (3_+{+}1){+}1,\,  4_+{+}1,\,  5_+\}.
\end{align*}
It is well-known that $|(\M_n)|=C_{n-1}$ where $C_n$ are the Catalan numbers defined as: 
\begin{align*} C_n=\frac{1}{n+1} { 2n \choose n}  \approx \frac{4^n} {(n+1)\sqrt{\pi n}},  \end{align*}
for every $n\geq 0$, see \cite{crepinvsek2009efficient, Dutton1986Comput}. For commodity we will write $c_n=C_{n-1}$. 

\begin{definition} \label{CharactFree}A magma $M$ is \emph{unidecomposable} iff for all $x,y,x',y'\in M$: 
\begin{align*}x+y=x'+y' \implies x=x' \mbox{ and } y=y'.\end{align*} 
\end{definition}
In the theory of monoids, Levi's characterization states that a graded monoid is free iff it is equidivisible, \cite{levi1944semigroups, Lallement1979Semigroups, sakarovitch2009elements}. We can prove an analogous result for unidecomposable magmas. The length and the uniqueness of decomposition allow us to define several notions and to prove some results by induction. 

\begin{theorem} A magma $M$ is free iff it is graded and unidecomposable.
\end{theorem}

\begin{proof}
$(\Rightarrow)$  $\M_X$ is graded by $\ell_X$ and the operation $x+y=(x,y)$ in $\M_X$ clearly makes $\M_X$ unidecomposable. Since $\M_X\cong M$ by assumption, $M$ is graded and unidecomposable.  
$(\Leftarrow)$ Let $\tilde{\ell}$ be the gradation of $M$. Consider the following sequence defined recursively: 
\begin{align*}
M_1&=M \setminus (M+M);\\
M_n&=\bigcup_{k=1}^{n-1} M_k+M_{n-k}, \mbox{ for } n>1.
\end{align*}
Let us see that this succession forms a partition of $M$. First we see $\bigcup_{k\geq 1} M_k=M$. The inclusion $\subseteq$ is trivial. For the other inclusion $\supseteq$ we see that for any $x\in M$ there is some integer $i$ such that $x\in M_i$ by induction on $\tilde{\ell}(x)$. If $\tilde{\ell}(x)=1$ then $x$ cannot be decomposed as $x=y+z$ for any $y,z\in M$, otherwise $\tilde{\ell}(x)=\tilde{\ell}(y)+\tilde{\ell}(z)=1$ which is impossible. Therefore, $x\in M_1$. Now we suppose that the statement is true for each $t\in M$ with $\tilde{\ell}(t)<n$ and let $x\in M$ such that $\tilde{\ell}(x)=n$. As above, if $x$ cannot be decomposed as $x=y+z$ then $x\in M_1$. If $x$ can be decomposed then $\tilde{\ell}(x)=\tilde{\ell}(y)+\tilde{\ell}(z)$. Then $\tilde{\ell}(y),\tilde{\ell}(z)<n$ and by hypothesis of induction $y\in M_i$ and $z\in M_j$ for some integers $i,j$. Therefore,  $y+z\in M_i+M_j$, and $x\in \bigcup_{k\geq 1} M_k$.

We see that $M_i\cap M_j=\emptyset$ for all $i\not=j$  
by induction on $n=\max\{i,j\}$. When $n=1$, then $i=1, j=1$, and then the statement is trivially true. Now we suppose that the statement is true for each integer $<n$, i.e. $M_i\cap M_j=\emptyset$, for all $i\not=j$ with $i,j<n$. 
For the sake of contradiction we suppose that there is an $x\in M$ such that $x\in M_k \cap M_{k'}\not=\emptyset$ with $n=\max\{k,k'\}>1$. If $k=1$ then $k'=n>1$, but this is not possible because elements in $M_k=M_1$ are not decomposable, while elements in $M_{k'}=M_n$ are decomposable, which means that $M_k\cap M_{k'} =\emptyset$ which is a contradiction. Similarly, if $k'=1$, then $k=n>1$. Therefore, we can assume $k,k'>1$ and then $x$ can be decomposed in two forms. That is, there are elements $a\in M_i,b \in M_j$ and  $a'\in M_{i'},b' \in M_{j'}$ such that $a+b=x=a'+b'$ with $i+j=k$, and $i'+j'=k'$, which means that $i,j,i',j'<n$ whereby we can apply the hypothesis of induction. Since by assumption $M$ is unidecomposable, $a=a'$ and $b=b'$. Then $M_i \cap M_{i'}\not=\emptyset$ and $M_j \cap M_{j'}\not=\emptyset$ which is a contradiction. 

Now we build the free magma $\M_X$, where $X=M_1$, and we consider the following succession of mappings:
\begin{align*}
&f_1: X_1 \longrightarrow M_1, \,\, f_1(x)=x, \forall x\in X_1\\
&f_n:X_n \longrightarrow M_n, \,\, f_{n}((x,y))=f_i(x)+f_j(y) \mbox{ such that }  (x,y) \in X_i \times  X_j.
\end{align*}
Then we build $f:\M_X \longrightarrow M$, as $f(x)=f_n(x)$ iff $x\in X_n$, which by definition is clearly a homomorphism of magmas. We see that it is indeed an isomorphism.  On the one hand, we see that each $f_n$ is injective by induction on $n$. $f_1$ is injective since it is the identity. We assume that $f_k$ is injective for each $k<n$, and notice that:
\begin{align*}
f_n(z)=f_n(z') &\implies f_n((x,y))=f_n((x',y')) \\ 
&\implies f_i(x)+f_j(y)=f_{i'}(x')+f_{j'}(y') \\
&\implies f_i(x)=f_{i'}(x'), f_j(y)=f_{j'}(y').
\end{align*}
Necessarily if $f_i(x)=f_{i'}(x')$ then $i=i'$, otherwise $M_i\cap M_{i'}\not=\emptyset$. Thus:
\begin{align*} f_i(x)=f_{i'}(x'), f_j(y)=f_{j'}(y') \implies f_i(x)=f_i(x'), f_j(y)=f_j(y').\end{align*}
By hypothesis of induction $f_i,f_j$ are injective, so $x=x'$ and $y=y'$, and $f_n$ is injective. Since every $f_n$ is injective, $f$ is injective. 
On the other hand, since every $f_n$ is surjective and $\bigcup_{k\geq 1} M_k=M$, we have that $\im f=\bigcup_{k\geq 1} \im f_k=\bigcup_{k\geq 1} M_k=M$, therefore $f$ is surjective.  
\end{proof}

\begin{example} The condition of being graded is necessary. Consider for example the trivial magma $M=\{m\}$ which is trivially unidecomposable since the equation $x+y=x'+y'$ have the unique solution $x=y=x'=y'=m$. However, $\{m\}$ is not graded, otherwise $\tilde{\ell}(m)=\tilde{\ell}(m+m)=\tilde{\ell}(m)+\tilde{\ell}(m)$ which is impossible since gradations must be positive. 

In fact, the trivial magma is the unique unidecomposable finite magma. If $M$ is finite and unidecomposable with $|M|>1$, then $|M|^2>|M|$, and then, by the pigeonhole principle, the operation sum $+:M \times M \longrightarrow M$ cannot be injective. This means that there are $(a,b), (a',b') \in M\times M$ with $(a,b)\not=(a',b')$ such that $a+b=a'+b'$, but since $M$ is unidecomposable, $a=a', b=b'$, which is a contradiction. Hence, $|M|=1$. 

A non-trivial example of unidecomposable but non-graded magma is the set of binary trees of finite or infinite depth where the operation $x+y$ consists also in joining the trees $x,y$ by a new root. The uniqueness of decomposition still holds for infinite trees. Notating $2^\infty$ the infinite full binary tree, we have an idempotent element $2^\infty + 2^\infty=2^\infty$ which disallows to define a gradation. 
\end{example}

\begin{example} \label{CartesianProd} The direct product $\M_X\times \M_Y$ is a free magma. We consider the gradation $\tilde{\ell}((x,y))=\ell_X(x)+\ell_Y(y)$. It is trivial to show that it is unidecomposable. We will see later a magma isomorphic to the direct product of two free magmas in Example~\ref{GeneratorProduct}. 
\end{example}

\section{The arithmetic of $\M$} \label{TheArithmetic}
\label{TheProduct}

\begin{definition} \label{DefProduct} The product on $\M$ is defined as the unique operation $\cdot: \M\times \M \longrightarrow \M$ such that for any $x,y,z\in \M$:
\begin{align*}
x\cdot 1 &=1\cdot x =x,\\
x\cdot (y+z)&=x\cdot y+x \cdot z.
\end{align*}
\end{definition}
It could be said that $(\M, +, \cdot,1)$ is a \emph{left semi-ringoid} with unity (a \emph{ringoid} is left- and right-distributive, see \cite[p.~206]{Rosenfeld1968Algebraic}). 
The product $x\cdot y$ can be regarded as the result of substituting each $1$ by $x$ into $y$. For example $2\cdot 3_+= 2\cdot (1+(1+1))=(2+2(1+1))=(2+(2+2))$, or in terms of trees: 
\begin{align*}2\cdot 3_+=\Big(\dos \Big) \,\,\, \cdot \,\,\, \Big( \,\trespositiu \Big)= \producte \, .\end{align*}
The following properties can be checked by induction on the length: 
\begin{align*}
\ell(xy)=\ell(x)\ell(y), \,\,\,\,\,\, x\leq y \implies ax \leq ay, 
\end{align*}
for each $x,y,a\in \M$. Let us note that every magma $M$ can be equipped with an external product $\circ: M \times \M \longrightarrow M$, holding the same conditions that in Definition~\ref{DefProduct}: $x{\circ} 1 =x$, $x{\circ} (a+b)=x{\circ} a+x{\circ} b$. In addition we have $x{\circ} (a {\cdot} b)=(x{\circ} a){\circ} b$ which makes $M$ a module-like structure. 

\begin{definition} Given $a,x \in \M$ we say that $a$ \emph{divides (at left)} $x$, notated $a|x$, iff there is some $y \in \M$ such that $x=ay$. We say that $x$ is \emph{prime} iff the only divisors of $x$ are $1$ and $x$ itself. We notate $\Prime$ the set of prime elements. The \emph{greatest common divisor} of $x,y \in \M$, notated $\gcd(x,y)$, is the greatest common leftmost divisor of $x$ and $y$.
\end{definition}
It can be proved by induction on the length that the product is associative, right and left cancellative and equidivisible. Thus, $(\M; \cdot)$ becomes a monoid with $1$ as identity. 
Every element $x \in \M$ has a unique factorization in prime elements, see \cite{blondel1995Structured}. 
We notice that if $n$ is a prime number then any element $x\in (\M)_n$ is prime. If $x$ was composite then $x=ab$ for some $a,b \in \M$, $a,b\not=1$. Applying $\ell$ we would have $n=\ell(x)=\ell(a)\ell(b)$ with $\ell(a),\ell(b)\not=1$ and then $n$ would be composite. We have, therefore, infinite prime elements and $(\M; \cdot)$ is an infinitely generated free monoid with the set of prime elements as generator: $(\M;\cdot)\cong (\Prime^*;\cdot)$. 

An elementary result of the ordinary arithmetic states that an integer $z>0$ is a prime number iff for all the sums $z=x+y$ with $x,y>0$ we have that $x$ and $y$ are coprime numbers (proof: if $x+y=z$, then  $\gcd(x,y)=\gcd(x,z-x)=\gcd(x,z)$; it is clear that $z$ is prime iff $\gcd(x,z)=1$ for each $x$, $0<x<z$).
In $\M$ the same characterization occurs as a particular case: the decomposition of $z=x+y$ is unique.  

\begin{theorem} \label{LemmaDivisio} For any $a,x,y \in \M$:
\begin{enumerate}
\item  $a| (x+y)$ iff either $a|x, a|y$ or $a=x+y$;
\item $x+y\in \Prime$ iff $\gcd(x,y)=1$. 
\end{enumerate}
\end{theorem}
\begin{proof} 
(1) $(\Leftarrow)$ trivial. $(\Rightarrow)$ If $a|(x+y)$, we have $x+y=az$ for some $z\in \M$. If $\ell(z)=1$, then $z=1$ and $x+y=a$, and we are done. If $\ell(z)>1$, then we can decompose $z=x'+y'$ and then $az=ax'+ay'$. Since this decomposition is unique and $az=x+y$, we have $ax'=x$ and $ay'=y$, i.e. $a|x$ and $a|y$.

(2) see an alternative proof in \cite{blondel1995Structured} for this characterization; here we remake it in a more arithmetical style. $(\Rightarrow)$ let $x+y \in \Prime$ and let us suppose that $\gcd(x,y)=d \not=1$. Then there are $x'$ and $y'$ such that $x=dx'$ and $y=dy'$. So $x+y=dx'+dy'=d(x'+y')$, whereby $x'+y' \not \in \Prime$, which is a contradiction. 
$(\Leftarrow)$ We suppose $x+y\not \in \Prime$ with $\gcd(x,y)=1$. There is some $d\in \M$, $d\not=1$ such that $d|(x+y)$. By point (1), either $d|x$ and $d|y$ or $d=x+y$. In the first case  $d | \gcd(x,y)$ so $\gcd(x,y)\not=1$ which is a contradiction; in the second case $x+y=d$, however if $d|(x+y)$ with $d\not=1$, then there is some $z \in \M$ with $z\not=1$ (otherwise $d=x+y$ would be prime) such that $x+y=dz$. Combining this with $x+y=d$ we have $d=dz$, that is, $z=1$, which is a contradiction. 
\end{proof}

The product allows us to characterize the endomorphisms of $\M$. We notate $\mathrm{End}(\M)$ the set of endomorphisms, and $\mathrm{Aut}(\M)$ the set of automorphisms. 

\begin{proposition} \label{HomoProperties} Some properties of the homomorphisms of $\M$ are the following:
\begin{enumerate}
\item Every endomorphism is of the form $f(x)=ax$ for some $a\in \M$. Thus:
\begin{align*}
(\mathrm{End}(\M); \circ)\cong (\M; \cdot) \cong (\Prime^*;\cdot).
\end{align*}
\item Every endomorphism is a monomorphism. 
\item For each endomorphism $f$ we have $f(xb)=f(x)b$, for any $b\in \M$.  
\item $\mathrm{Aut}(\M)=\{\mathrm{Id}\}$. 
\end{enumerate}
\end{proposition}
\begin{proof} (1) by induction on the length, taking $a=f(1)$. (2) $f(x)=f(y)$ means that $ax=ay$. Since the the product is cancellative, $x=y$. So every endomorphism is a monomorphism. (3) $f(xb)=a(xb)=(ax)b=f(x)b$.   
(4) $\ell(f(x))=\ell(ax)=\ell(a)\ell(x)\geq \ell(a)$. If $\im(f)=\M$, in particular $1\in \im(f)$, and then $f(x)=1$ for some $x\in \M$. Taking $\ell$ we have $1=\ell(1)=\ell(f(x))\geq \ell(a)$, whereby $a=1$. Thus, if $f(x)=ax$ is an epimorphism, $a=1$. Hence, $\mathrm{Aut}(\M)=\{\mathrm{Id}\}$. 
\end{proof}

We notate $\Pi_n$ the number of prime elements in $\M$ with length $n$, $\Pi_n=|(\Prime)_n|$. The first values of $\Pi_n$, $n\geq 1$, are: 
\begin{align*} 
0,\, 1,\, 2,\, 4,\, 14,\, 38,\, 132,\, 420,\, 1426,\, 4834,\, 16796,\, 58688,\, 208012, \ldots,
\end{align*}
see \cite{Sloan2018TheEncy}.
A first formula for $\Pi_n$ is: 
\begin{align*} 
\Pi_n=c_{n} - \sum_{ \stackrel{1<d<n}{ d|n}} \Pi_{d} \cdot c_{\frac{n}{d}}\: . 
\end{align*}
The right-hand side of the formula subtracts the number of composite elements from $c_n$. An element in $\M$ is composite iff it has some prime proper divisor and from there, the formula (see \cite{Sloan2018TheEncy, Amerlynck1998Iterees}).
We show a more closed expression for $\Pi_n$.  The number of prime factors counting multiplicity of the integer $n$ is notated $\Omega(n)$, for example, $\Omega(2^3 \cdot 5^4)=7$.

\begin{proposition} \label{CountingPrimes} For each $n>1$. If $n$ is prime $\Pi_n=c_n$. When $n$ is composite:
\begin{align*}
\Pi_n =\sum_{s=1}^{\Omega(n)}   \sum_{\stackrel{1<k_1,\ldots,k_s < n}{k_1\cdots k_s=n}}  (-1)^{s+1} c_{k_1}\cdots c_{k_s}.
\end{align*}
\end{proposition}
\begin{proof} Clearly if $n$ is prime $\Pi_n=c_n$, so we suppose $n$ composite. We define the \emph{partial Dirichlet convolution product} of two integer sequences $a_n,b_n$ as $(a\star b)_n=\sum_{1<i,j<n, \, ij=n} a_i b_j$, where the sum is null when there are no $i,j$ such that $i\cdot j=n$ with $1<i,j<n$ (we say ``partial'' because the extrems $a_1b_n$ and $a_nb_1$ are not counted;  see, for example, \cite{knuth1989concrete} for Dirichlet products). Given $m$ integer sequences  $(a_1)_n, \ldots, (a_m)_n$ their product is:
\begin{align*}
(a_1\star \cdots \star a_m)_n=\sum_{\stackrel{1<k_1,\ldots,k_s <n}{k_1\cdots k_s=n}} (a_1)_{k_1}\cdots (a_m)_{k_s},
\end{align*}
whereby this operation is associative, and in addition, it is distributive with the sum of sequences.
The $m$th power of $a_n$ is notated $(a^{\star m})_n$. Thus, the  formula above for $\Pi_n$ can be written as $\Pi_n=c_n - \Pi_n \star c_n$. By substituting this into itself $N$ times we get (we abbreviate $\Pi=\Pi_n, c=c_n$):
\begin{align*} 
\Pi&=c - \Pi \star c,\\
\Pi&=c-c^{\star 2} + \Pi\star c^{\star 2}, \\
\Pi&=c-c^{\star 2} + c^{\star 3} - \Pi\star c^{\star 3}, \\
\Pi&=c-c^{\star 2} + c^{\star 3} - c^{\star 4} + \Pi\star c^{\star 4}, \\
\,\, \vdots \\
\Pi&=\Big(\sum_{s=1}^N (-1)^{s+1} c^{\star s} \Big) + (-1)^{N} \Pi\star c^{\star N}.
\end{align*} 
Now we see that if $N=\Omega(n)$ then $(\Pi\star c^{\star N})_n=0$. This is because the sum in the expansion of this convolution product,
\begin{align*}
(\Pi \star c^{\star N})_n=\sum_{ \stackrel{1<t, k_1, \ldots, k_N <n} {t\cdot k_1\cdots k_N=n}} \Pi_{t}c_{k_1} \cdots c_{k_N},
\end{align*}
requires $N=\Omega(n)+1$ proper factors, which is impossible, since $n$ has at most $\Omega(n)$ proper factors, and then the sum is necessarily null. Hence, $\Pi=\sum_{s=1}^{\Omega(n)} (-1)^{s+1} c^{\star s}$, where $(c^{*s})_n=\sum_{\stackrel{1<k_1,\ldots,k_s < n}{k_1\cdots k_s=n}}  c_{k_1}\cdots c_{k_s}$. 
\end{proof}

\begin{example} By using Proposition~\ref{CountingPrimes}, some particular values of $\Pi_n$ can be calculated. For $p,q,r$ different prime numbers: 
\begin{align*}
\Pi_{p}&=c_p,\\
\Pi_{pq} &=c_{pq}-2c_pc_q,\\
\Pi_{p^2} &=c_{p^2}-c_p^2,\\
\Pi_{p^2q} &=c_{p^2q}-2(c_pc_{pq}+c_qc_{p^2})+3c_p^2c_q,\\
\Pi_{pqr} &=c_{pqr}- 2(c_pc_{qr}+c_qc_{pr}+c_rc_{pq})+6c_pc_qc_r,\\
\Pi_{p^3} &=c_{p^3}- 2c_pc_{p^2}+c_p^3.
\end{align*}
\end{example}

\begin{proposition} Prime elements are abundant in $\M$ in the sense that:
\[\lim_{n\rightarrow \infty} \frac{\Pi_n}{c_n}=1.\] 
\end{proposition}
\begin{proof} First we calculate the limit $\lim_{n\rightarrow \infty} \frac{(\Pi \star c)_n }{c_n}=0$. Let us observe the following simple fact. If $pq=n$ and $1<p,q <n$, then $p,q \leq \lfloor\frac{n}{2} \rfloor$. In addition, if we suppose that $p,q > \lfloor\frac{n}{3} \rfloor$, then $n=pq> (\lfloor\frac{n}{3}\rfloor )^2$. Since $\lfloor\frac{n}{3}\rfloor \geq \frac{n}{4}$ we have $n> (\frac{n}{4})^2$, that is $n<16$. Therefore, if $n\geq 16$, then either $p \leq \lfloor\frac{n}{3} \rfloor$ or $q \leq \lfloor\frac{n}{3} \rfloor$. 
Now we see the inequalities:
\begin{align*} 0 \leq (\Pi \star c)_n \leq (c\star c)_n = \sum_{\stackrel{1<p,q<n}{pq=n}} c_pc_q \leq  n c_{\lfloor \frac{n}{2}\rfloor} c_{\lfloor \frac{n}{3}\rfloor},
\end{align*}
for any $n\geq 16$, where we have used that $\Pi_n \leq c_n$. Consider the following inequalities for the Catalan numbers, see \cite{Dutton1986Comput}:
\begin{align*} \frac{4^n}{(n+1)\sqrt{\pi n \frac{4n}{4n-1}}}<C_n<\frac{4^n}{(n+1)\sqrt{\pi n \frac{4n+1}{4n}}}.
\end{align*}
Since for $n\geq 4$ we have $\sqrt{\pi n\frac{4n}{4n-1}}\leq n \leq n+1$, we can weaken these bounds as:
\begin{align*} \frac{4^n}{(n+1)^2}<C_n<4^n.
\end{align*}
Therefore, for any $n\geq 16$, and recalling that $c_n=C_{n-1}$ and that $\lfloor x \rfloor \leq x$: 
\begin{align*} 0 \leq \frac{(\Pi \star c)_n}{c_n} \leq  
\frac{n c_{\lfloor \frac{n}{2}\rfloor} c_{\lfloor \frac{n}{3}\rfloor}}{c_n}  \leq n^3 \cdot  4^{\lfloor\frac{n}{2}\rfloor+\lfloor\frac{n}{3}\rfloor-n-1}\leq \frac{n^3}{4^{\frac{n}{6}+1}},
\end{align*}
and taking limits we have that $\lim_{n\rightarrow \infty} \frac{(\Pi \star c)_n}{c_n}=0.$ Finally, from the expression $\Pi_n=c_n - (\Pi\star c)_n$, dividing it by $c_n$ and taking limits we have:
\begin{align*}
\lim_{n\rightarrow \infty} \frac{\Pi_n}{c_n}= 1 - \lim_{n\rightarrow \infty} \frac{(\Pi \star c)_n}{c_n}=1-0=1. \,\,\,\,\,\,\,\,\,\,\,\,\,\,\,\, \qedhere
\end{align*} 
\end{proof}

\section {Submagmas and generators} \label{SubmagmasAndGenerators}

In this section we visit a little zoo of submagmas of $\M$ and sublattices of $\Sub(\M)$. Given a magma $M$ and $N,N' \in \Sub(M)$ we notate $N \vee N'=\langle N\cup N'\rangle$.
Since $\langle \cdot \rangle$ is a closure operator, $(\Sub(M); \vee, \cap)$ is a complete lattice with bounds $\emptyset$ and $M$ (see \cite[p.~21]{sankappanavar1981course} or \cite[p.~48, 147]{davey2002introduction}). Given $N\in \Sub(M)$, $\Sub(N)$ is a sublattice of $\Sub(M)$.

\begin{proposition} Let $N$ be a non-empty submagma of $\M$ and $a\in \M$. 
\begin{align*}
Na \subseteq N,  & \mbox{ but in general } Na \notin \Sub(\M);\\
aN \in \Sub(\M), & \mbox{ but in general } aN\not \subseteq N.
\end{align*}
In addition $aN\cong N$. 
\end{proposition}
\begin{proof} We see $Na \subseteq N$ by induction on the length of $a$. When $\ell(a)=1$, $a=1$, and $N\cdot 1=N$. If $\ell(a)>1$ and $xa \in Na$, then we can decompose $a=a'+a''$, and then $xa=xa'+xa''$ is in $N$ ($N$ is a submagma and $xa', xa'' \in N$ by hypothesis of induction). In general $Na$ is not a submagma, e.g. $\langle 3_+ \rangle 2 \notin \Sub(\M)$. Now we see  $aN \in \Sub(\M)$. Let $ax,ay \in aN$. Then $x,y\in N$ which implies that $x+y\in N$, and then $a(x+y)=ax+ay\in aN$. In general $aN\not \subseteq N$, e.g. $2 \langle 3_+ \rangle  \not \subseteq  \langle 3_+ \rangle $.
Finally, we saw in Proposition~\ref{HomoProperties} that $f(x)=ax$ is always a monomorphism, so $f:N \longrightarrow aN$ is an isomorphism. 
\end{proof}

\begin{example} The simplest non-empty submagmas of $\M$ are of the form $a\M = \{ y  \in \M \mid\, a | y\}=\langle a \rangle$. If $ax,ay \in a\M$ then, by the distributivity of the product, $ax+ay=a(x+y) \in a\M$. 

$\M\setminus \{1\}$ is a submagma of $\M$. If $x,y\not=1$, $\ell(x),\ell(y)>1$, and then $\ell(x+y)=\ell(x)+\ell(y)>2$, so $x+y\not=1$. Clearly $\M\setminus \{1\}$ is the unique maximal proper submagma of $\M$. 
\end{example}

\begin{example} [Longitudinal submagmas]
It is trivial to show that if $f: M \longrightarrow N$ is a homomorphism of magmas and $M', N'$ are submagmas of $M$ and $N$ respectively, then $f(M')$ is a submagma of $N$,
and $f^{-1}(N')$ is a submagma of $M$. 
The length $\ell: \M \longrightarrow \mathbb{N}$ is a homomorphism of magmas (from a free magma to a semigroup) which permits defining several submagmas. We say that a submagma $N$ is \emph{longitudinal} iff $N=\ell^{-1}(N_0)$ for some subsemigroup $N_0 \subseteq \mathbb{N}$.  
$\M$ and $\emptyset$ are longitudinal. Longitudinal submagmas form a complete sublattice of $\Sub(\M)$. 

Since $\{ n \in \mathbb{N} \mid n\geq i\}$ is a subsemigroup, $M_i=\ell^{-1}(\{ n \in \mathbb{N} \mid n\geq i\})$ is a submagma. Such submagmas form a complete sublattice of that of the longitudinal submagmas, which is distributive: $M_i\cap M_j=M_{\min\{i,j\}}$, $M_i \vee M_j=M_{\max\{i,j\}}$. 

More interesting are the submagmas $\ell^{-1}(\alpha \mathbb{N})= \{ x\in \M \mid \ell(x)\equiv 0 \,\, \mathrm{mod}(\alpha)\}$, for each $\alpha\geq 1$. Notice that $a\M \subsetneq \ell^{-1}(\ell(a)\mathbb{N})$, and this is the least longitudinal submagma containing $a\M$. This can be generalized:  
given a submagma $N\subseteq \M$ we have that
$N \subseteq \ell^{-1} \ell (N)$,
and $\ell^{-1} \ell (N)$ is the least longitudinal submagma containing $N$. 
Let us see it. Trivially $\ell^{-1} \ell (N)$ is a submagma. We suppose another longitudinal submagma $\ell^{-1}(N_0)$ such that: $N \subseteq \ell^{-1}(N_0) \subseteq \ell^{-1} \ell (N)$. Then $\ell(N) \subseteq \ell\ell^{-1}(N_0) \subseteq \ell\ell^{-1} \ell (N)$, that is $\ell(N)\subseteq N_0 \subseteq \ell(N)$, therefore $\ell(N)=N_0$. 
\end{example}

\begin{definition} Given  a non-empty submagma $N\subseteq \M_X$, consider the mapping $g_N: N \longrightarrow 2^N$ defined recursively as:
\begin{align*}
g_N(z)=\begin{cases} g_N(x) \cup g_N(y) &\mbox{ if } \exists x,y\in N \mbox{ such that } z=x+y; \\
\{z\} & \mbox{ otherwise. }\end{cases} 
\end{align*}
We define $\G(N)=\bigcup_{x\in N} g_N(x)$ when $N$ is non-empty, and $\G(\emptyset)=\emptyset$. 
The \emph{rank} of a submagma $N$ is the number $\rk(N)=|\G(N)|$ when the cardinality is finite. 
\end{definition}

\begin{lemma}  Given a non-empty submagma $N \subseteq \M_X$, $\G(N)$ is the unique minimal generating set of $N$. The pair of mappings $\langle \cdot \rangle: 2^{\M_X} \longrightarrow \Sub(\M_X)$, $\G: \Sub(\M_X) \longrightarrow 2^{\M_X}$ form a Galois connection.
\end{lemma}
\begin{proof} $g_N(z)$ is well defined since the $x,y\in N$ in the decomposition $x+y=z$ are unique. In addition, $\ell_X(y),\ell_X(z) <\ell_X(x)$, so the recursion must finish at some point.
We see that $\langle \G(N) \rangle=N$. $(\subseteq)$ $\G(N)\subseteq N$, then $\langle \G(N) \rangle \subseteq N$. $(\supseteq)$ For each $x\in N$, $x\in \langle g_N(x)\rangle$, so $N  \subseteq \bigcup_{x\in N} \langle g_N(x) \rangle \subseteq \langle \bigcup_{x\in N} g_N(x) \rangle =\langle \G(N) \rangle$.

We see that $\G(N)$ is minimal. If we had $a\in \G(N)$ such that $\langle \G(N)\setminus \{a\}\rangle=N$, then $a$ should be generated in $\G(N)\setminus \{a\}$, and in particular there would be some $x,y\in N$ such that $a=x+y$. However, by assumption $a\in \G(N)$ and, by definition of $g_N$, this decomposition is not possible. Hence, $\G(N)$ is minimal. 

Finally, we see that $\G(N)$ is the unique minimal generating set. Suppose that $\langle A\rangle = N$ for some $A \subseteq N$. Notice that $\G(N) \subseteq A$  (this is because in order to generate an element $x$ we need the additive factors given by $g_N(x)$). However, since  $\G(N)$ is minimal, and we have supposed $A$ is, $A=\G(N)$. 

Using that $\langle \G(N)\rangle =N$ and that $\G(\langle A \rangle ) \subseteq A$, we have that $M \subseteq \langle A \rangle$ iff $\G(M) \subseteq A$ which is the definition of a Galois connection, see \cite{davey2002introduction}.
\end{proof}

\begin{lemma} \label{SubmagmaCondition} Given a non-empty submagma $N \subseteq \M_X$ we have that $x+y \in N$ iff either $x,y\in N$ or $x+y\in \G(N)$.
\end{lemma}
\begin{proof}
The direction $(\Leftarrow)$ is trivial. For the direction $(\Rightarrow)$ we take $x+y\in N$ and we suppose that either $x$ or $y$ is not in $N$. Then, by definition of $g_N$, we have $g_N(x+y)=\{x+y\}$, so $x+y \in \G(N)$, and we are done. Notice that we saw this as a particular case in Theorem~\ref{LemmaDivisio}(1), by putting $x+y \in a\M$ instead of $a|x+y$.
\end{proof}

\begin{example} \label{ExampleMaximalSubmagma}Consider again the maximal proper submagma $\M\setminus \{1\}$. By applying $g_N$ we can find the minimal generating set. We have that an element $z\in N=\M\setminus \{1\}$ can be decomposed into $z=x+y$ with $x,y \in N$ provided that $\ell(x),\ell(y)\geq 2$. The elements which cannot be decomposed (and thus they are generators, $g_N(z)=\{z\}$) are those such that $\ell(x)=1$ or $\ell(y)=1$, that is, elements of the form $1+\M$ or $\M+1$:
\begin{align*}
\G(\M\setminus \{1\})=(1+\M)\cup (\M+1),
\end{align*}
and therefore, $\M\setminus \{1\}$ has infinite rank.
\end{example}

\begin{example} The generator set of $\ell^{-1}(\alpha \mathbb{N})$ can be calculated as: 
\begin{align*}
\G(\ell^{-1}(\alpha \mathbb{N}))=\bigcup_{ \stackrel{\alpha | p+q,} { \alpha  {\not \, |} \, p, \, \alpha {\not \, | } \, q}} (\M)_p+(\M)_q.
\end{align*}
In particular for $\alpha=2$:
{\small \begin{align*} 
\G(\ell^{-1}(2 \mathbb{N}))&=(\M)_1+(\M)_1 \\
&\cup\,\, (\M)_1+(\M)_3 \,\, \cup \,\, (\M)_3+(\M)_1 \\
&\cup\,\, (\M)_1+(\M)_5\,\,\cup \,\,(\M)_3+(\M)_3 \,\,\cup\,\, (\M)_5+(\M)_1 \\
&\cup\,\, (\M)_1+(\M)_7\,\,\cup \,\,(\M)_3+(\M)_5 \,\,\cup \,\,(\M)_5+(\M)_3 \,\,\cup \,\,(\M)_7+(\M)_1\\
&\,\,\,\vdots
\end{align*}}
\end{example}

\begin{example} [Symmetric submagmas] We define recursively the \emph{symmetric element of} $x$, notated $\overline{x}$,  by $\overline{1}=1$, and $\overline{x+y}=\overline{y}+\overline{x}$. Thus, for example, $\overline{n_+ }=n_-$, $\overline{n_- }=n_+$.
We have that $\overline{xy}=\overline{x} \, \overline{y}$ and that every antimorphism ($f(x+y)=f(y)+f(x)$) is of the form  $f(x)=a\overline{x}$, for some $a\in \M$. 

If $N\subseteq \M$ is a submagma, $\overline{N}$ is a submagma, and $\rk(\overline{N})=\rk(N)$. We say that a submagma $N\subseteq \M$ is \emph{symmetric} iff $\overline{N}=N$. 
 Since $\overline{N \vee M}=\overline{N} \vee \overline{M}$ and $\overline{N\cap M}=\overline{N}\cap \overline{M}$ the set of symmetric submagmas is a complete sublattice. Every symmetric submagma is of the form $N \vee \overline{N}$.
An \emph{strongly symmetric submagma} is a $\vee$-sum (maybe empty) of submagmas of the form $2^n\M$. 
Every symmetric submagma $N$ can be decomposed as
$N=S \vee T$ with $S\cap T=\emptyset$, where $T$ is a strongly symmetric submagma, and $S$ is symmetric but does not contain any strongly symmetric submagma. When $N$ has finite rank, $\rk(S)$ is even, i.e.:
\begin{align*}
N= S\vee T= \langle a_1, \overline{a_1},  a_2, \overline{a_2},  \ldots, a_r, \overline{a_r} \rangle \vee \langle  2^{n_1}, \ldots, 2^{n_s} \rangle,
\end{align*}
for some integers $r,s\geq 0$. 
\end{example}

\begin{example} We can see that $\rk(\langle a,b\rangle)=2$ iff $a {\not |} \, b$ and $b {\not |} \, a$. $(\Rightarrow)$ if $a |b$ then $b=ax$ for some $x \in \M$. Then, $\langle a,b \rangle=\langle a,ax \rangle=a \langle 1, x \rangle= a \langle 1 \rangle =\langle a \rangle$ which is a contradiction. Similarly, if $b|a$. $(\Leftarrow)$ if $\rk(\langle a,b \rangle)<2$, $\langle a,b \rangle=\langle c \rangle$ for some $c\in \M$. On the one hand, since  $\langle a,b \rangle \subseteq \langle c \rangle$, there are some $x,y\in \M$ such that $a=cx, b=cy$. On the other hand, $\langle a,b \rangle \supseteq \langle c \rangle$, which means that there is some combination of $a$'s and $b$'s equalling $c$, and then $\ell(c)=p\ell(a)+q\ell(b)$ for some integers $p,q \geq 0$. Therefore, we have:  
\begin{align*}
\ell(c)=p\ell(cx)+q\ell(cy)=p\ell(c)\ell(x)+q\ell(c)\ell(y) \implies 1=p\ell(x)+q\ell(y).\end{align*} 
This implies that either $a=cx=c1=c$ or $b=cy=c1=c$. In the first case  $a|b $, in the second one, $b|a$. 
\end{example}

\begin{lemma} \label{LemmaS} Each subset of the set $S=\{ 2, 3_+,4_+,5_+, \ldots \}$ is a minimal generating set. 
\end{lemma}
\begin{proof} Consider $S'\subseteq S$, and $r_+ \in S'$. We prove that $r_+$ cannot be generated in $S'\setminus \{r_+\}$, which implies that $r_+$ cannot be generated in any subset $S'' \subseteq S'\setminus \{r_+\}$, and therefore, $S'$ is minimal. We decompose $S'\setminus \{r_+\}$ into two complementary subsets:
\begin{align*}
A_r=\{ n_+ \in   S'\setminus \{r_+\} \mid 1<n<r\} \mbox{ and } B_r=\{ n_+ \in   S'\setminus \{r_+\} \mid n>r\}.
\end{align*}
Let us see that $r_+\not\in \langle A_r \rangle$. 
Clearly $r_+\notin A_r$. Now we suppose that $r_+\in \langle A_r \rangle \setminus A_r$. By Lemma~\ref{SubmagmaCondition} there are $x,y \in \langle A_r  \rangle$ such that $r_+=x+y$. Since $r\geq 2$, $r_+=1+(r-1)_+=x+y$ which means that $x=1 \in \langle A_r \rangle$ which is absurd by definition of $A_r$. 

Now we see that if we add some elements from $B_r$, they do not contribute to generate $r_+$. This is simply because for any element $x\in B_r$, $\ell(x)>r=\ell(r_+)$, so any combination of elements of $B_r$ and elements of $A_r$ will have length $>r$. So $r_+\not \in \langle A_r \cup B_r \rangle =\langle S' \setminus \{r_+\} \rangle.$
Since this holds for any $r_+ \in S'$ and any $S' \subseteq S$, the statement is proved. 
\end{proof}

We define the following submagmas:
\begin{align*}
\N_n^+=\langle 2, 3_+,4_+,5_+, \ldots, n_+\rangle \mbox{ and } \N_\infty^+=\langle 2, 3_+,4_+,5_+, \ldots \rangle.
\end{align*}
By the Lemma~\ref{LemmaS}, $\rk(\N_n^+)=n-1$ and $\rk(\N_\infty^+)=\infty$.  
One can compare the following theorem to that of Nielsen-Schreier \cite{stillwell1993Classical} which states that free groups only contain free subgroups. It is also well-known that free groups of rank at least two contain any other countably generated free group. Similar statements are possible regarding the cyclic free magma. 

\begin{theorem} \label{NonCountability} Each submagma of $\M$ is free, and $\M_X$ can be embedded in $\M$ for any countable set $X$. In addition, the set of submagmas $\Sub(\M)$ is not countable.
\end{theorem}
\begin{proof}  By Theorem~\ref{CharactFree} a magma is free iff it is graded and unidecomposable. This occurs for each submagma $N\subseteq \M$ since the operation in $N$ is the same that $\M$ and we can inherit the gradation $\ell$ restricted to $N$.  

Let $\M_X$ be a free magma. If $X$ is finite, with $|X|=n$, we take a bijection $X \longrightarrow \G(\N_{n+1}^+)$; if $X$ is not finite we take a  bijection $X \longrightarrow \G(\N_{\infty}^+)$. These mappings extend to the isomorphisms $\M_X \longrightarrow \N_{n+1}^+$ and $\M_X \longrightarrow \N_{\infty}^+$, respectively.

For the third statement we define the mapping $\varphi: \mathbb{N} \longrightarrow S$, $\varphi(n)=(n+1)_+$. With this we define the mapping:
\begin{align*}
2^\mathbb{N} &\longrightarrow \Sub(\M)\\
X &\longmapsto \langle S \setminus \varphi (X) \rangle,
\end{align*}
and we only have to prove that it is an injective mapping. If $\langle S \setminus \varphi(X) \rangle=\langle S \setminus \varphi(Y) \rangle$ then $\G(\langle S \setminus \varphi(X) \rangle)=\G(\langle S \setminus \varphi(Y) \rangle)$. By Lemma~\ref{LemmaS} we have that $S \setminus \varphi(X)$ and S $\setminus \varphi(Y)$ are minimal generating sets which means that $S \setminus \varphi(X)=S \setminus \varphi(Y)$, and then $\varphi(X)=\varphi(Y)$. Since $\varphi$ is injective, $X=Y$. 
\end{proof}

\begin{example} \label{GeneratorProduct} We saw in Example~\ref{CartesianProd} that the direct product of two free magmas is a free magma. The generator set of $\M^2=\M\times\M$ can be calculated as $\G(\M^2)=(\{1\} \times \M) \cup (\M\times \{1\})$, and then it has infinite rank. There is a natural isomorphism  
$f: \M^2 \longrightarrow \M\setminus \{1\}$ 
defined over the generators by $f((1,y))=1+y$ and $f((x,1))=x+1$. But  notice that in general $f((x,y))\not=x+y$; for example, $f((2^2,2))=3_-\cdot2\not=2^2+2=2\cdot3_-$. 

More in general, $\G(\M_X \times \M_Y)=(\M_X\times Y) \cup (X\times \M_Y) $. By Theorem~\ref{NonCountability}, we have the isomorphism: 
\begin{align*}
\M_X \times \M_Y \cong \M_{(\M_X\times Y) \cup (X\times \M_Y) }.
\end{align*}
\end{example}

\begin{note} Since every free magma $\M_X$ (even, every finite direct product $\M_X\times \M_Y$) inhabits in $\M$, for a lot of properties on free magmas, the study of the cyclic free magma $\M$ suffices. 
\end{note}

\section{Ideals and principal ideals} \label{IdealsAndPrincipalIdeals}

Given a submagma $N \subseteq \M$, we say that it is a \emph{(right multiplicative) ideal} iff for each $a\in N$ and for each $x \in \M$ we have $ax \in N$. 
It is trivial to show that $a\M=\langle a \rangle$ is an ideal of $\M$, which we call \emph{principal ideal}. Thus, every submagma $N=\langle A \rangle$ can be put as a $\vee$-sum of principal ideals: 
$N = \bigvee_{a\in A} a\M$.
It turns out that in $\M$ all submagmas are straightforward ideals, though, non necessarily principal. 

\begin{proposition} Each submagma of $\M$ is an ideal. 
\end{proposition}
\begin{proof} Given $N\in \Sub(\M)$ we pick $a\in N$. For every $x \in \M$, we prove by induction on the length of $x$ that $ax\in N$. If $\ell(x)=1$  then $x=1$, and $ax=a1=a\in N$. For the general case $\ell(x)>1$ we can decompose $x=y+z$ with $\ell(y),\ell(z)<k$. Then $ax=a(y+z)=ay+az$. Since $ay,az \in N$ by hypothesis of induction and $N$ is a submagma, $ax \in N$. 
\end{proof}

In the ring $\mathbb{Z}$ the identity of Bezout states that $a\mathbb{Z} + b\mathbb{Z}=\gcd(a,b)\mathbb{Z}$. In $\M$ we enjoy a less restricted version. 
Let us see how principal ideals behave under $\subseteq, \cap$ and $\vee$.

\begin{theorem} \label{PrincipalIdeals} We have that:
\begin{enumerate}
\item $a\M \subseteq b\M \iff b | a$.
\item $a\M \cap b\M =\begin{cases} b\M & \mbox{ if } a|b, \\
a\M & \mbox{ if } b|a, \\
\emptyset &\mbox{ otherwise.}  \end{cases}$
\item 
$a\M \vee b\M \subseteq\gcd(a,b)\M$ and $\gcd(a,b)\M$ is the least principal ideal containing $a\M \vee b\M$.
\end{enumerate}
\end{theorem}
\begin{proof} (1) $a\M \subseteq b\M$ iff for each $ax\in a\M$ there is a $by \in b\M$ such that $ax=by$. Since this must be satisfied for any $x$, necessarily $b$ divides $a$. 

(2) an element in $a\M\cap b\M\not=\emptyset $ must satisfy that $ax=by$ for some $x,y\in \M$. Then by the equidivisibility property we have that either $a|b$ or $b|a$. In the first case $b\M \subseteq a\M$, so $a\M \cap b\M=b\M$. The case $b|a$ is symmetric. 

(3) let $d=\gcd(a,b)$. We factorize $a=da'$, and $b=db'$ for some $a',b'\in \M$. We have $a\M\vee b\M=da'\M \vee db'\M=d(a\M \vee b\M) \subseteq d\M$. Now we have to see that if we have the inclusions: $a\M \vee b\M \subseteq d'\M \subseteq d\M \mbox{ then } d'\M=d\M$. 
From the first inclusion we have that for any $x,y\in \M$ there must be a $z\in \M$  such that $ax+by=d'z$, i.e. $d'|(ax+by)$. But by Theorem~\ref{LemmaDivisio}(1) we know that either $d'|ax$ and $d'|bx$ or $d'=ax+by$. In the first case we observe that $d'|\gcd(a,b)=d$. However, from the second inclusion, we have that $d|d'$, whereby $d=d'$ and we are done. In the second case we have $d'=ax+by$ but notice that this must be satisfied for any other different $x',y'\in \M$. So $d'=ax'+bx'$, whereby $ax+by=ax'+by'$, and then $ax=ax', bx=bx'$ which implies that $x=x', y=y'$ which is absurd. 
\end{proof}
\begin{note}
The points (1) and (2) say that the set of principal ideals is a semi-lattice with the intersection (we assume that $\emptyset$ is principal). In particular, (1) implies that a principal ideal is maximal in this semilattice iff it is of the form $p\M$ with $p\in\Prime$. 
This result is also proposed in \cite[p.~161]{Bourbaki1989Algebra} without defining a product operation, where prime elements are called \emph{primitives} and the principals ideals are called \emph{monogenous}. By our approach we have preferred to highlight the arithmetical nature of $\M$.   
\end{note}

\begin{corollary} For any set $X$, $\Sub(\M_X)$ is not modular. 
\end{corollary}
\begin{proof} Consider the elements $2,3_+$, and $(2+3_+) \in \M$. Since $2$ and $3_+$ are prime, by Theorem~\ref{PrincipalIdeals}(2), $2\M \cap 3_+\M=\emptyset$. Every element of $3_+\M \vee (2{+}3_+)\M$ is a combination of $3_+$ and $2{+}3_+$, and we can see by induction on the length that no such element is divisible by $2$. Hence, $2\M \cap \big(3_+\M \vee (2{+}3_+)\M\big)=\emptyset$. Finally, we have that  $2\M \vee \big(3_+\M \vee (2{+}3_+)\M\big)=2\M \vee 3_+\M$, since $2{+}3_+ \in 2\M \vee 3_+\M$. 
In sum, we have the so-called pentagon lattice as a sublattice of $\Sub(\M)$, see Figure~\ref{Pentagon}. A lattice is non-modular iff it contains a copy of such lattice \cite{davey2002introduction}, whereby $\M$ is not modular. Since every free magma $\M_X$ contains a copy of the cyclic free magma $\M$, $\M_X$ is not modular. 
\end{proof}

\begin{figure}[tb] 
\begin{align*} \small
  \xymatrix@C-1.5pc@R-0.5pc { & 2\M \vee 3_+\M  \ar@{-}[ddl] \ar@{-}[dr]& \\
 && 3_+\M \vee (2{+}3_+)\M \ar@{-}[d] \\
   2\M  \ar@{-}[dr] && 3_+\M  \ar@{-}[dl] \\
  & \emptyset} 
  \end{align*}
\caption{Pentagon sublattice in $\Sub(\M)$.}\label{Pentagon}
\end{figure}
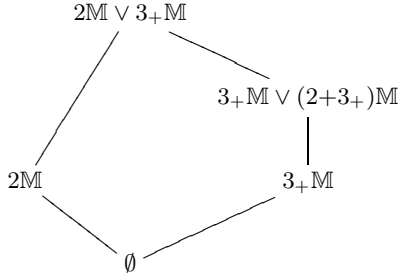

\section{$k$-maximal submagmas and additive prime sets} \label{MaximalAndAdditivePrime}

\begin{definition} Given a complete lattice $\mathcal{L}$ with an upper bound $U$ and a lower bound $V$, we say that $X\in \mathcal{L}$ is \emph{$k$-maximal} iff the longest ascending chain from $X$ to $U$ has length $k$, i.e. there are only $k-1$ elements between $X$ and $U$. 
$k$-minimal elements are defined symmetrically for the lower bound.  
\end{definition}

\begin{example} $\M$ is a $0$-maximal submagma. $\M\setminus \{1\}$ is $1$-maximal, because we have the ascending chain $\M \setminus \{1\} \subsetneq \M$, and we cannot interpose another submagma between.  Since we have the monomorphism $f:\M^2 \longrightarrow M\setminus \{1\} \subset \M$, Example~\ref{GeneratorProduct}, if $M$ is $k$-maximal submagma in $\M^2$, then $f(M)$ is  $(k+1)$-maximal in $\M$.

With the exception of the empty submagma, $\M$  has no $k$-minimal submagmas. Given a submagma $N \not=\emptyset$ we can always pick a proper submagma: $a{\cdot} 2 {\cdot} \M \subsetneq a\M \subseteq N$, where $a\in \G(N)$. 
\end{example}

In general $\M$ does not satisfy the \emph{ascending chain condition}. We have, for example, the infinite chain: $\N_2^+ \subset \N_3^+ \subset \N_4^+ \subset \cdots$, whereby $\N_2^+=2\M$ is not $k$-maximal for any integer $k$. So, some submagmas are $k$-maximal while other not. This means that a Hasse diagram does not capture the complete structure of $\Sub(\M)$, since infinite chains cannot be depicted. Notwithstanding, the study of $k$-maximal submagmas and the Hasse diagram, is visually helpful to understand the structure of submagmas of $\M$. 

\begin{definition} Given a magma $M$, we say that a set $\p \subseteq M$ is an \emph{additive prime set} iff
\begin{align*}
x+y \in \p \implies x\in \p \mbox{ or } y \in \p.
\end{align*}
Where the context permits, we say simply that $\p$ is prime.
We say that a set $\s \subseteq M$ is \emph{closed} iff
\begin{align*}
x+y \in \s \implies x\in \s \mbox{ and } y \in \s.
\end{align*}
Clearly, every closed set is a prime set.
\end{definition}

\begin{example} \label{ExamplePrimeSets} $\emptyset, \{1\}, \{1,3_-,5_+\}$, $\{1,2,2^2, 3_-\}$ and $\{1,2, 2^2, 2^3,2^4\}$ are some examples of prime sets in $\M$. Even, $\emptyset, \{1\}$, $\{1,2,2^2, 3_-\}$ and $\{1,2, 2^2, 2^3,2^4\}$ are closed subsets. However, the set $\{1,3_-,5_+\}$ is not closed because $3_-=2+1$ but $2$ is not in the set. Closed subset arises in Lazard set of trees to form a basis of a free Lie algebra, see \cite{Reutenauer1993BookFreeLie}. 
\end{example} 

\begin{example}
Prime sets allows us to do another proof of the non-countability of $\Sub(\M)$. We just consider the injective mapping:
\begin{align*}
2^\mathbb{N} &\longrightarrow \Sub(\M)\\
X &\longmapsto \big( \{1\} \cup \{(n+1)_+ \mid n \in X\} \big)^\com,
\end{align*}
and it is checked that $\{1\} \cup\{(n+1)_+ \mid n \in X\}$ is a prime set. 
\end{example}

\begin{proposition} \label{BasiPropPrime} Let $\p,\q$ be subsets of a magma $M$. 
\begin{enumerate}
\item $\p$ is prime $\iff \p^\com$ is a submagma of $M$.
\item $\p,\q$ prime $\implies \p\cup \q$ prime.
\item $\p,\q$ prime $\implies \p\cup \q \cup (\p+\q)$ prime.
\item Each non-empty prime set contains at least one generator of $M$.
\item If $\p\subseteq \M$ is prime with $\p\not=\emptyset$, then  $1\in \p$.
\end{enumerate}
\end{proposition}
\begin{proof} (1) is rather trivial: $\p^\com$ is a submagma iff for any $x,y\in \p^\com$, $x+y\in  \p^\com$, which is equivalent to say that if $x,y\not \in \p$, $x+y \not \in \p$, which is equivalent to say that if $x+y \in \p$, $x\in \p $ or $y \in \p$.   
(2) If $\p,\q$ are prime, then $\p^\com, \q^\com$ are submagmas. $\p^\com \cap \q^\com=(\p \cup \q)^\com$ is a submagma which means that $\p\cup \q$ is prime.
(3) we suppose that $x+y\in (\p\cup \q) \cup (\p+\q)=\mathfrak{r}$. Then either $x+y \in (\p\cup \q)$ or $x+y \in (\p+\q)$. In the first case, by (2), either $x\in \p\cup \q$ or $y \in \p\cup \q$, which implies that $x\in \mathfrak{r}$ or $y \in \mathfrak{r}$. In the second case $x+y \in \p+\q$, which means that $x\in \p$ and $y\in \q$. Then, trivially $x\in \mathfrak{r}$ or $y \in \mathfrak{r}$. 
(4) we suppose that $\G(M)\cap \p=\emptyset$ with $\p$ prime. Necessarily $\G(M) \subseteq \p^\com$ which implies that $\langle \G(M)\rangle \subseteq \langle \p^\com \rangle$. Since $\langle \G(M)\rangle=M$ and $\langle \p^\com\rangle=\p^\com$, we have that $M \subseteq \p^\com \subseteq M$, whereby  $M=\p^\com$, or what is the same $\p=\emptyset$. (5) is a particular case of (4) since $1$ is the generator of $\M$.
\end{proof}

\begin{proposition}  Let $\s,\tc$ be subsets of a magma $M$.  
\begin{enumerate}
\item $\s$ is closed iff $\s^\com$ is a bilateral ideal under the sum, that is: $M+\s^\com \subseteq \s^\com$ and $\s^\com + M \subseteq \s^\com$.
\item If $\s,\tc$ are closed, $\s \cup\tc$ is closed.
\item Let $M=\M$. For $x\in \M$, we define the \emph{segment} $[x]=\{y  \in \M \mid y\leq x\}$. We have that $\s$ is closed iff $[\s]=\s$. In particular, $[X]$ is a closed set for any set $X\subseteq \M$. 
\item If $M=\M$ and $\s\not=\emptyset$ is closed, $\s^\com$ is not finitely generated.
\end{enumerate}
\end{proposition}

\begin{proof} 
(1) by definition of closed set, if $x\not \in \s$ or $y\not \in  \s$, then $x+y \not \in \s$. In particular, if $x\not \in \s$ and $y\in M$, then $x+y \not \in \s$, and if $x\in M$ and $y \not \in \s$, then $x+y \not \in \s$. The first statement says that  $\s^\com +M \subseteq \s^\com$, the second one, $M + \s^\com \subseteq \s^\com$.

(2) by (1) $\s^\com,\tc^\com$ are bilateral ideals. Then $M+(\s\cup \mathfrak{t})^\com=M+(\s^\com \cap \mathfrak{t}^\com)=(M+\s^\com) \cap (M+\mathfrak{t}^\com) \subseteq \s^\com \cap \tc^\com=(\s\cup \tc)^\com$. The inclusion $(\s\cup \mathfrak{t})^\com+ M \subseteq (\s\cup \tc)^\com$ is similar, therefore $(\s\cup \tc)^\com$ is a bilateral ideal. 

(3) from the definition, $\s$ is closed iff for each $x \in \s$ both additive factors of $x$ are also in $\s$. By applying again the definition, $\s$ is closed iff the additive factors of $x$ and the additive factors of the additive factors of $x$ are in $\s$, and so forth. All these factors form the set $[x]$. So $\s$ is closed iff $[\s]\subseteq \s$. Since we have always that $[\s] \supseteq \s$, we have that $\s$ is closed iff  $[\s]= \s$. 

It is easily checked that the set extension $[\cdot]: 2^\M \longrightarrow 2^\M$ is a closure operator. In particular, $[X]$ is a closed set for any subset $X\subseteq \M$, since $[[X]]=[X]$.

(4) suppose that $N=\s^\com \subseteq\M$ is finitely generated with generators $a_1, \ldots, a_n$. Consider then the element:
\begin{align*}
b=(\cdots ((a_1+a_2)+a_3)+\cdots +a_n))\cdots) \in N,
\end{align*}
and let us show that $b+1\not \in N$. By Lemma~\ref{SubmagmaCondition} if $b+1\in N$ then either $b,1 \in N$ or $b+1 \in \G(N)=\{a_1, \ldots, a_n\}$. In the first case, if $1\in N$ then $N=\M$ and $\s=\emptyset$, which is a contradiction.  
In the second case, $b+1=a_j$ for some $j=1, \ldots,n $. This is impossible since $\ell(b+1)=1+\sum_{i=1}^n \ell(a_i) > \ell(a_j)$ for each $j=1, \ldots,n $.
Since $b\in N$, $1\in \M$, but $b+1\not\in N$, then $N$ is not a bilateral ideal under $+$, which is a contradiction with (1). 
\end{proof}

\begin{definition}
We call the \emph{spectrum of a set} $X\subseteq \M$ the set $\Spec(X)=\{ \p \subseteq X  \mid \p \mbox{ prime}\}$. 
\end{definition}
Given a prime set $\p$, by Proposition~\ref{BasiPropPrime}(2), $\Spec(\p)$ is a semilattice with the union operation and bounds $\emptyset, \p$. 
We recall some vocabulary on order theory. A partial order $(\mathcal{L}; \leq)$ is said to have \emph{length} $k$, notated $\length(\mathcal{L})=k$, iff its longest chain has size $k+1$, and it is said to have \emph{width} $k$, notated $\width(\mathcal{L})=k$, iff its greatest antichain has size $k$. 
Figure~\ref{SpectraExample} shows some spectra of the prime sets from the Example~\ref{ExamplePrimeSets} with lengths $3, 4$, and $5$, respectively; and with widths $2,2$ and $1$, respectively.  We abbreviate ``$a \in \mathcal{L}$ is a maximal element under the order $\leq$'' as ``$a$ is $\leq$-maximal''.

\begin{figure}[tb] 
{\small
\begin{align*} 
  \xymatrix@C-3.0pc@R-0.5pc { & \{1,3_-,5_+\}    \ar@{-}[dl] \ar@{-}[dr]& \\
 \{1,3_-\}  \ar@{-}[dr] & & \{1,5_+\}  \ar@{-}[dl] \\
  &\{1\} \ar@{-}[d] &  \\
  &\emptyset &} 
    \,\,\,\,\,\,\,\,\,\,\,\,\,\,\,\,\,\,
   \xymatrix@C-3.0pc@R-0.5pc { & \{1,2, 2^2,3_-\}    \ar@{-}[dl] \ar@{-}[dr]& \\
 \{1,2, 2^2\}  \ar@{-}[d] & & \{1,2, 3_-\}  \ar@{-}[dll]\ar@{-}[d] \\
   \{1,2\}  \ar@{-}[dr] & & \{1, 3_-\}  \ar@{-}[dl] \\
  &\{1\} \ar@{-}[d] &  \\
  &\emptyset &} 
  \,\,\,\,\,\,\,\,\,\,\,\,\,\,\,\,\,\,
  \xymatrix@C-3.0pc@R-1.0pc {  \{1,2, 2^2, 2^3, 2^4\}    \ar@{-}[d]& \\ \{1,2, 2^2, 2^3\}    \ar@{-}[d]& \\
 \{1,2, 2^2\}    \ar@{-}[d]& \\
  \{1,2\}    \ar@{-}[d]& \\
    \{1\}    \ar@{-}[d]& \\
  \emptyset }
\end{align*}
}
\caption{Spectra from the Example~\ref{ExamplePrimeSets}.}\label{SpectraExample}
\end{figure}
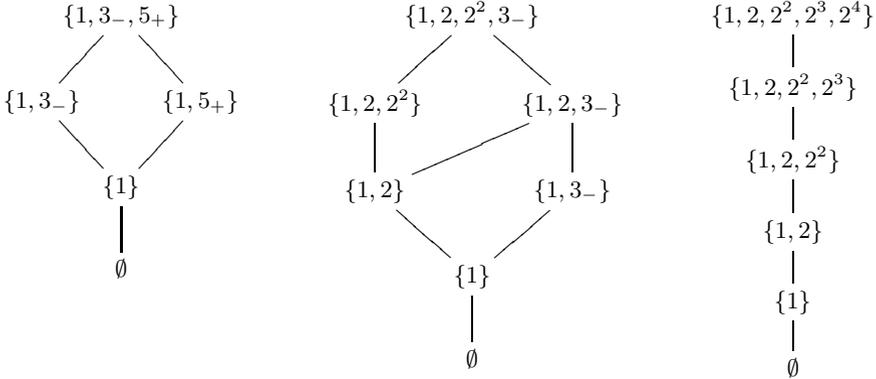

\begin{lemma} \label{LemaLong} Let $\p\subset \M$ be a finite prime set. Then $\length(\Spec(\p))=|\p|$. 
\end{lemma}
\begin{proof} First we see that $\length(\Spec(\p))\leq|\p|$. Since $\Spec(\p) \subseteq 2^\p$, the longest chain in $\Spec(\p)$ is also in $2^\p$, and therefore $\length(\Spec(\p)) \leq \length(2^\p)$. Now we see that $\length(2^\p) \leq |\p|$. Consider a chain in $2^\p$, say $\emptyset=A_0\subset A_1 \subset \cdots \subset A_m=\p$. We have that $|A_{i+1}\setminus A_i|\geq 1$, otherwise the inclusion is not strict. So for each inclusion in this chain we have at least a new element. Since $\p$ has $|\p|$ elements, $m \leq |\p|$.  

We see the other inequality, $\length(\Spec(\p))\geq|\p|$, by induction on the size $|\p|$. If $|\p|=0$, $\p=\emptyset$ and $\length(\Spec(\p))=0$.  If $|\p|=1$, then $\p=\{1\}$ and $\length(\Spec(\p))=1$.  
We suppose that the statement is true for each prime set with size $n$ and let $\p$ be a prime set with $|\p|=n+1$. We pick some $\leq$-maximal element, say $a\in \p$. It turns out that $\p \setminus \{a\}$ is a prime set. Let us see this. For any $x+y \in \p\setminus \{a\}$ we have that either $x \in \p$ or $y\in \p$. If $x=a$ or $y=a$ then $x+y \rightarrow a$, and then $x+y \geq a$, but $a$ is maximal so $x,y\not=a$. Thus, if $x+y \in \p\setminus \{a\}$ then either $x\in \p\setminus \{a\}$ or $y \in \p \setminus \{a\}$, whereby $\p\setminus \{a\}$ is a prime set. 
Now consider the longest chain in $\Spec(\p \setminus \{a\})$, which, by hypothesis of induction, will have length at least $n$, say: $\q_1 \subset \q_2 \subset \cdots \subset \q_n=\p \setminus \{a\}.$
We can expand this chain in $\Spec(\p)$ by adding $\p$ as: $\q_1 \subset \q_2 \subset \cdots \subset \q_n \subset \p$. Clearly the last one is a proper inclusion: $\q_n = \p\setminus \{a\} \subset \p$. So this chain has length at least $n+1$.
\end{proof}

\begin{theorem} \label{KMaximalArePrime} Let $N\subseteq \M$ a submagma. $N$ is $k$-maximal iff  $|N^\com|=k$. 
\end{theorem}
\begin{proof} We have an isomorphism of semilattices: 
\begin{align*}
(\cdot)^\com: (\Sub(\M); \cap) \longrightarrow (\mathbf{Spec}_{+}(\M); \cup).
\end{align*}
It is clear that $N$ is a $k$-maximal submagma iff $\p=N^\com$ is a $k$-minimal prime set. It is also clear that $\p$ is $k$-minimal in $\mathbf{Spec}_{+}(\M)$ iff $\length(\Spec(\p))=k$, since every descending chain in $\mathbf{Spec}_{+}(\M)$ from $\p$ to $\emptyset$  is in $\Spec(\p)$, and conversely. Finally, by Lemma~\ref{LemaLong}, $\length(\Spec(\p))=|\p|$. 
\end{proof}

\section{All the finite prime sets} \label{AllTheFinitePrimeSets}

Theorem~\ref{KMaximalArePrime} explains that in order to know the $k$th level of the Hasse diagram of $\Sub(\M)$ it suffices to know all the prime sets of size $k$. 
The smallest prime sets can be calculated manually. We use the notation minus to reverse the order, that is $x-y=y+x$. With this, the expression $x\pm y$ means the two possibilities $x+y, y+x$ with the convention that when a term $x\pm y$ appears twice, or more times, in a set we will suppose that the sign is the same in all the instances. For example, $\{1, a\pm 1, b\pm (a\pm 1)\}$ represents four alternative sets: $\{1, a{+}1, b{+}(a{+}1)\}$, $\{1, a{-}1, b{+} (a{-}1)\}$, $\{1, a{+}1, (a{+}1){+}b\}$, and $\{1, a{-}1, (a{-}1){+}b\}$.

\begin{proposition} \label{SmallPrimeSets} Let $\p \subseteq \M$ be a prime set, with $|\p|=n$, then:
\begin{enumerate}
\item if $n=0$ then $\p=\emptyset$; 
\item if $n=1$ then $\p=\{ 1\}$;
\item if  $n=2$ then $\p= \{1, a\pm 1\}$ for some $a\in \M$; 
\item if  $n=3$ then $\p=\{1,a\pm 1, b\pm 1\}$ or $\p=\{1, a\pm 1, b\pm (a\pm 1)\}$ for some $a,b \in \M$. 
\end{enumerate}
\end{proposition}
\begin{proof} (1) is trivial. (2) is trivial, by Proposition~\ref{BasiPropPrime}(5). For (3) we suppose $\p=\{1,x\}$, with $x\not=1$, and then $x=y+z$ and $y \in \p$ or $z\in \p$. Since $y,z\not=x$, then either $y=1$ or $z=1$. That is, there is an $a\in \M$ such that $x=1+a$ or $x=a+1$. 
(4) for $n=3$, we pick some $\leq$-maximal element, say $x\in \p$. Then $\p\setminus \{x\}$ is a prime set (we saw this in the proof of Lemma~\ref{LemaLong}) with $|\p \setminus \{x\}|=2$. So $\p \setminus \{x\}=\{1, a\pm 1\}$ for some $a\in \M$. We decompose $x=y+z$, and then either $y \in \{1, a \pm 1\}$ or $z \in  \{1, a \pm 1\}$. This yields four cases: (a) if $y=1$ then $\p=\{1, a \pm 1, 1+z\}$; (b) if $y=a\pm 1$, then  $\p=\{1, a \pm 1, (a\pm 1)+z\}$; (c) if $z=1$ then $\p=\{1, a \pm 1, y+1\}$; (d) if $z=a\pm 1$ then $\p=\{1, a \pm 1, y+ (a\pm 1)\}$. 
The cases (a) and (c) can be summarized as $\p=\{1,a\pm 1, b\pm 1\}$ for some $b\in \M$. The cases (b) and (d), as $\p=\{1, a\pm 1, b\pm (a\pm 1)\}$ for some $b \in \M$. 
\end{proof}

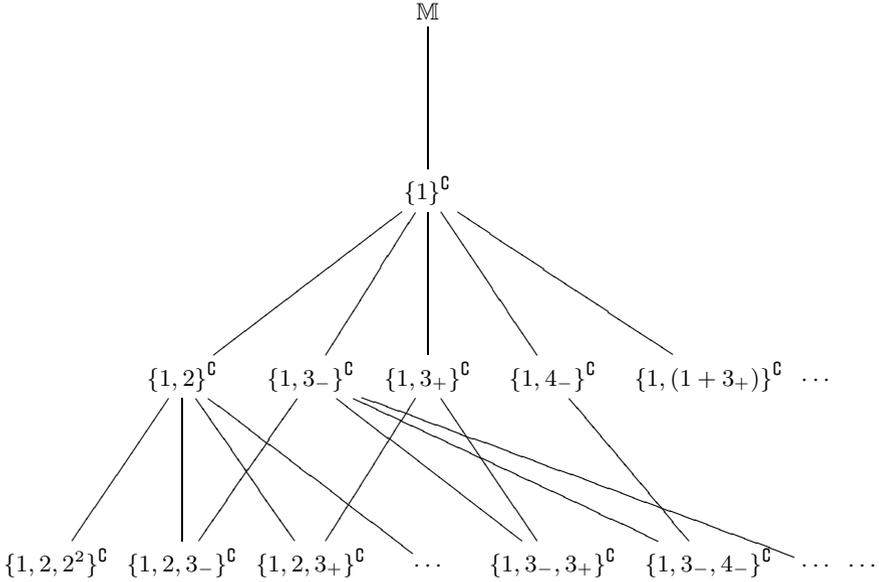
\begin{figure}[tb] 

\begin{align*}
\small \xymatrix@C-1.9pc@R4.5pc{ & & & \M  \ar@{-}[d] & & & \\
 & & & \{1\}^\com \ar@{-}[dll] \ar@{-}[dl] \ar@{-}[d] \ar@{-}[dr] \ar@{-}[drr]& & & \\
 & \{1,2\}^\com \ar@{-}[dl] \ar@{-}[d] \ar@{-}[dr] \ar@{-}[drr]& \{1,3_-\}^\com \ar@{-}[dl]  \ar@{-}[drr] \ar@{-}[drrr] \ar@{-}[drrrr] & \{1,3_+\}^\com  \ar@{-}[dl]  \ar@{-}[dr]& \{1,4_-\}^\com  \ar@{-}[dr] & \{1,(1+3_+)\}^\com  & \cdots \\
\{1,2,2^2\}^\com & \{1,2,3_-\}^\com & \{1,2,3_+\}^\com & \cdots & \{1,3_-,3_+\}^\com & \{1,3_-,4_-\}^\com & \cdots  & \cdots 
}
\end{align*}
\caption{The first four levels of the Hasse diagram of $\Sub(\M)$}\label{HasseDiagram}
\end{figure}

The four first levels of the Hasse diagram of $\Sub(\M)$ are depicted in Figure~\ref{HasseDiagram}. 
In order to get all the prime sets, the way in the above proof is not useful. However, notice that prime sets of size three are of two distinct forms: $\{1,a\pm 1, b\pm 1\}$ and $\{1, a\pm 1, b\pm (a\pm 1)\}$
for some $a,b \in \M$. We show a general procedure to obtain all the prime sets by grouping them by ``families''. 

We fix the set of strings given by the free monoid  $\Sigma^*$, where $\Sigma$ is the infinite alphabet $\Sigma=\{a, b, c, \ldots\} \cup \{(, ), \pm, 1\}$, see  \cite{hopcroft1979introduction}. We call the letters $a,b,c, \ldots$ \emph{variables}.
We define recursivelly a \emph{term} as:
\begin{align*}
&1, a, b, c, \ldots \mbox{ are terms};\\
&\mbox{if } x,y \mbox{ are terms, } (x\pm y) \mbox{ is a term}.
\end{align*}
We notate $\mathscr{L}\subset \Sigma^*$ the set of terms, for example, $((b\pm (a\pm 1))\pm c)$ is a string of $\mathscr{L}$. We give an interpretation of the language $\mathscr{L}$ by the following definition. 

\begin{definition} A \emph{substitution} is a mapping $\sigma:\mathscr{L}\longrightarrow \M$ such that:
\begin{enumerate}
\item $\sigma(1)=1$,
\item $\sigma((x \pm y))=\sigma(x)+\sigma(y)$ or $\sigma(x)- \sigma(y)$, provided $x,y$ are terms. 
\end{enumerate}
\end{definition}

A substitution stays defined by a choice of values in $\M$ for the variables and a choice of signs $+,-$ for each occurrence of the symbol $\pm$. We notice that this definition is consistent with the convention above for the symbol $\pm$. For example, if consider the term $(b\pm (a\pm 1))$, a possible substitution is given by the values $\sigma(a)=2$ and $\sigma(b)=3_-$ and the sign $+$ for the first token of $\pm$ and $-$ for the second token, which yields $\sigma((b\pm (a\pm 1)))=3_+{+}(2{-}1)$. When we consider a set of terms $\{1, a\pm 1, (b\pm (a\pm 1))\}$, the substitution yields the set $\{ 1, 2{-}1, 3_+{+}(2{-}1)\}$. That is, $\sigma$ keep up the sign $-$ for the term $2{-}1$ in both instances.  

\begin{definition} \label{Labelization}  Let $T$ be a rooted tree with root $r$. We define recursively a \emph{labelization} $\Lab: T \longrightarrow \mathscr{L}\subseteq \Sigma^*$ as follows:
\begin{enumerate}
\item $\Lab(r)=1$;
\item if $x$ is a vertex already labelled, and $y$ is a child of $x$, then $\Lab(y)=\big(a \pm  \Lab(x) \big)$, where $a$ is a new variable never used before. 
\end{enumerate}
Given a tree $T$ and a labelization $\Lab$ we define a \emph{family} as: 
\begin{align*} 
\Lambda_{\Lab}(T)= \{ & \sigma(\Lab(T))  \mid \sigma \mbox{ is a substitution}\}.
\end{align*}
\end{definition}

We remark that $\Lambda_{\Lab}(T)$ does not depend on the choice of the variables in each step of the labelization. Different labelizations yields the same family, or more formally, if $\Lab$ and $\Lab'$ are different labelizations, there are substitutions $\sigma$ and $\sigma'$ such that $\sigma(\Lab(T))=\sigma'(\Lab'(T))$, and then $\Lambda_{\Lab}(T)=\Lambda_{\Lab'}(T)$. Therefore, we notate families just as $\Lambda(T)$.  
The following theorem offers a representation of prime sets by rooted trees. 

\begin{theorem} \label{CaracteritzacioPrime} $\p \subset \M$ is a finite prime set iff there is a rooted tree $T$ with $|T|=|\p|$ such that $\p \in \Lambda(T)$. 
\end{theorem}
\begin{proof} 
$(\Rightarrow)$ We consider the digraph with vertices $\M$ and edges  $(z,z')$ given by $z\rightarrow z'$.  We notate it by $\mathcal{R}_{\rightarrow}$. 
Given a prime set $\p$ we consider the induced  subgraph of $\mathcal{R}_{\rightarrow}$ over $\p$, notated $\mathcal{R}_\rightarrow (\p)$. Notice that, since $\p$ is prime, for each element $x\in \p$ there is a path from $x$ to 1 in $\mathcal{R}_\rightarrow (\p)$, or what is the same, every vertex is reachable from $1$ by inverting the arrows, see Figure~\ref{Papallona}. We take a \emph{spanning arborescence} (a spanning tree for directed graphs) of $\mathcal{R}_\rightarrow (\p)$, say $T_\p$. 

By the above comment such arborescence exists (see \cite{Zwick2018TheEncy, Gabow1986Efficient}) and necessarily the root of $T_\p$ is the element $1$, and notice that this is the first condition, Definition~\ref{Labelization}(1), of the labelization of the root. 
In addition, we have that $(y,x)$ is an edge of $T_\p$ iff $y\rightarrow x$, that is, iff there is some $a\in \M$ such that $y=a+x$ or $y=x+a$, abbreviated, $y=a \pm x$. This is just the second condition, Definition~\ref{Labelization}(2), of the labelization of the vertex $y$: $\Lab(y)=\big(a \pm  \Lab(x) \big)$. Hence, $\p \in \Lambda(T_\p)$. 

$(\Leftarrow)$ Suppose $\p \in \Lambda(T)$. We prove by induction on the size $|T|$ that $\p$ is prime. If $|T|=1$ then $\Lambda(T)=\{\{1\}\}$, and then $\p=\{1\}$, which is prime. We assume that the statement is true for trees of size $n$ and we consider a tree $T$ with $|T|=n+1$. We prune a leave $v$ of $T$, which yields a subtree $T'$ for which the statement is true. Therefore, there must be a prime set $\p' \in \Lambda(T')$ which yields $\p$ by adding an extra element. By the labelization and a substitution, the effect of adding the leave $v$ to $T'$ consists in joining an element to $\p'$ of the form $a\pm x$ where $a \in \M$ and $x\in \p'$. Since $\p'$ is prime, $\p$ is prime.  
\end{proof}

\begin{figure}[tb] 
\begin{align*}
\small \xymatrix{&&  & 2^4 \ar[d]  & & &\\
&2{+}3_- \ar@/_0.9pc/[ddr] \ar@/^2.3pc/[rrdd]  &  & 2^3 \ar[d]  & & 3_+{+}2   \ar@/^0.9pc/[ddl] \ar@/_2.3pc/[lldd] &\\
&& 3_-{+}2  \ar@/_0.3pc/[d] \ar@/^0.3pc/[rd]  & 2^2  \ar[d] &  2{+}3_+  \ar@/^0.3pc/[d] \ar@/_0.3pc/[ld] &&\\
3_-{\cdot}2^2 \ar[r] &3_-{\cdot}2 \ar[r] & 3_-  \ar[r] \ar@/^0.3pc/[rd]  & 2  \ar[d] & 3_+\ar[l] \ar@/_0.3pc/[ld] & 3_+{\cdot}2 \ar[l] & 3_+{\cdot}2^2 \ar[l]\\
&4_- \ar@/^0.5pc/[ru]  \ar@/_1.3pc/[rr] & 1{+}3_- \ar[r]\ar@/^0.3pc/[u]  & 1 & 3_+{+}1 \ar[l]\ar@/_0.3pc/[u] & 4_+\ar@/_0.5pc/[lu]  \ar@/^1.3pc/[ll]&\\
& & 1{+}4_- \ar@/^0.5pc/[ul] \ar@/_0.5pc/[ur]  && 4_+{+}1 \ar@/^0.5pc/[ul] \ar@/_0.5pc/[ur] &&\\
& 5_- \ar@/^0.7pc/[uu] \ar@/_2.5pc/[uurr] & & & & 5_+ \ar@/_0.7pc/[uu] \ar@/^2.5pc/[uull] & }
\end{align*}
\caption{A fragment of the digraph $\mathcal{R}_{\rightarrow}$}\label{Papallona}
\end{figure}

Figure~\ref{Papallona} shows a fragment of the digraph $\mathcal{R}_{\rightarrow}$ mentioned in the proof of Theorem~\ref{CaracteritzacioPrime}. We observe that it is anti-reflexive, anti-symmetric, and acyclic. The out-degree of a vertex $x$ is two, provided that $x \notin \M{\cdot}2$. When $x\in \M{\cdot}2$, the out-degree is one if $x\not=1$, and zero if $x=1$. 

\begin{example} There are nine families of prime sets of size five given by the nine rooted trees up to isomorphisms:
\begin{center}
{\fontsize{3}{4}\selectfont
\begin{forest}
for tree={grow'=90, circle,draw,l sep=1pt, s sep=4pt}
[[ ] [] [] [] ]
\end{forest}
}
\,\,\,\,\,\,\,\,
{\fontsize{3}{4}\selectfont
\begin{forest}
 for tree={grow'=90, circle,draw,l sep=1pt, s sep=4pt}
[[ [] ] [] []  ]
\end{forest}
}
\,\,\,\,\,\,\,\,
{\fontsize{3}{4}\selectfont
\begin{forest}
 for tree={grow'=90, circle,draw,l sep=1pt, s sep=4pt}
[[ [] ] [[] ]  ]
\end{forest}
}
\,\,\,\,\,\,\,\,
{\fontsize{3}{4}\selectfont
\begin{forest}
 for tree={grow'=90, circle,draw,l sep=1pt, s sep=4pt}
[ [ [] [] ]  []   ]
\end{forest}
}
\,\,\,\,\,\,\,\,
{\fontsize{3}{4}\selectfont
\begin{forest}
 for tree={grow'=90, circle,draw,l sep=1pt, s sep=4pt}
[ [[] [] [] ] ]
\end{forest}
}
\,\,\,\,\,\,\,\,
{\fontsize{3}{4}\selectfont
\begin{forest}
 for tree={grow'=90, circle,draw,l sep=1pt, s sep=4pt}
[ [ [[][]] ] ]
\end{forest}
}
\,\,\,\,\,\,\,\,
{\fontsize{3}{4}\selectfont
\begin{forest}
 for tree={grow'=90, circle,draw,l sep=1pt, s sep=4pt}
[[[ [] ] [] ] ]
\end{forest}
}
\,\,\,\,\,\,\,\,
{\fontsize{3}{4}\selectfont
\begin{forest}
 for tree={grow'=90, circle,draw,l sep=1pt, s sep=4pt}
[[[[]]] [] ]
\end{forest}
}
\,\,\,\,\,\,\,\,
{\fontsize{3}{4}\selectfont
\begin{forest}
 for tree={grow'=90, circle,draw,l sep=1pt, s sep=4pt}
[[[[[]]]] ]
\end{forest}
}
\end{center}
For instance, the fourth tree $T$ defines the family of prime sets:
\begin{align*}
\Lambda(T)=\{ \{1, a\pm1, b\pm 1, c\pm (a\pm 1), d\pm (a\pm 1)\} \mid a, b, c, d  \in \M\}.
\end{align*}
\end{example}

The representation of prime sets through rooted trees give us information about the width of $\Spec(\p)$. We finish with the following lemma: 

\begin{lemma} \label{NumberLeaves} If $T$ is a rooted tree and $\p\in \Lambda(T)$:
\begin{align*}
\width(\Spec(\p)) \geq \mbox{number of leaves of } T.
\end{align*}
\end{lemma}
\begin{proof} 
Let $T$ be a tree and let $\p\in \Lambda(T)$, such that $\p=\sigma(\Lab(T))$ for some labelization $\Lab$ and substitution $\sigma$. 
A rooted tree $T$ can be naturally partially ordered from the relation of descendent, which we notate $(T; \preceq)$. Actually, the mapping $\sigma \circ \Lab: (T; \preceq) \longrightarrow (\p;\leq)$ is bijective and preserves the orders, although it is not an isomorphism. 

For each $y \in \p$ we notate the set $\p_y=\{x \in \p \mid x \leq y\}$. It is easy to prove that $\p_y \in \Spec(\p)$ and that $\p_x \subseteq \p_y$ iff $x\leq y$. 
Now let us note that if $(\sigma\circ \Lab)^{-1}(x)$ and $(\sigma\circ \Lab)^{-1}(y)$ are leaves of the tree $T$ (i.e. they are $\preceq$-maximal elements of $(T;\preceq)$), then, since $\sigma\circ \Lab$ preserve the orders and it is surjective, $x$ and $y$ are $\leq$-maximal in $\p$. Since maximal elements are not comparable, $\p_x$ and $\p_y$ are not comparable. 
This means that the set of the $\p_x$ such that $(\sigma\circ \Lab)^{-1}(x)$ is a leave forms an antichain in $\Spec(\p)$. Since $\sigma\circ \Lab$ is bijective, $\width(\Spec(\p))\geq $ number of leaves.  
\end{proof}

\begin{example} Let us consider the case of the ``slimmest'' prime sets. We say that a prime set $\p$ is \emph{slim} when $\width(\Spec(\p))=1$. First we observe that slim prime sets are join-irreducible, see \cite{davey2002introduction}. Consider the decomposition $\p=\q \cup \rp$. Since the width is one, $\Spec(\p)$ only contains antichains of size one. Then $\q \subseteq \rp$ or $\rp \subseteq \q$. In the first case $\p=\q\cup \rp=\rp$, in the second case $\p=\q\cup \rp=\q$.

We can represent all slim prime sets. Let $\p \in \Lambda(T)$ with $\p$ slim. By Lemma~\ref{NumberLeaves}, $T$ has at most one leave, therefore $T$ must be a path tree. By applying a labelization from Definition~\ref{Labelization} to a path tree we obtain that slim prime sets of size $n$ are of the form: 
\begin{align*}
\{1, a_1\pm 1, a_2 \pm (a_1 \pm 1),  \ldots, a_n \pm (a_{n-1} \pm ( \cdots ( a_1 \pm 1)\cdots)\},
\end{align*}
for some $a_1, \ldots, a_n \in \M$. However, not all prime sets of such form are slim. Consider for example $\p=\{1,2,3_-, 2{+}3_-\}$ for which we have the antichain: $\{1,3_-,2{+}3_-\}$ and $\{1,2, 2{+}3_- \}$. 
Regarding closed sets, one can see easily that the unique slim closed sets must be of the form $[2^n]=\{1, 2, 2^2, \ldots, 2^n\}$ for some $n\geq 0$. 
\end{example}


\end{document}